\documentclass[11pt]{amsart}
\usepackage{geometry}                
\usepackage{graphicx}
\usepackage{url}
\usepackage{amssymb}
\usepackage{amsrefs}
\usepackage{epstopdf}
\usepackage{tikz}
\usetikzlibrary{shapes,patterns,calc,snakes}
\newcommand\no{100}

\newcommand\mbb{\mathbb}

\newcommand\N{\mbb{N}}

\newcommand\R{\mbb{R}}
\newcommand\Z{\mbb{Z}}

\theoremstyle{plain}
\newtheorem{Thm}[equation]{Theorem}
\newtheorem{Prop}[equation]{Proposition}

\newtheorem{Lemma}[equation]{Lemma}

\newtheorem{Conjecture}[equation]{Conjecture}

\newtheorem*{Thm*}{Theorem}
\newtheorem*{Prop*}{Proposition}
\newtheorem*{Cor*}{Corollary}
\newtheorem*{Lemma*}{Lemma}
\newtheorem*{Sublemma*}{Sublemma}
\newtheorem*{Conjecture*}{Conjecture}

\theoremstyle{definition}

\newtheorem{Def}[equation]{Definition}

\newtheorem{Example}[equation]{Example}

\newtheorem*{Construction*}{Construction}
\newtheorem*{Def*}{Definition}
\newtheorem*{Defs*}{Definitions}
\newtheorem*{Example*}{Example}
\newtheorem*{Examples*}{Examples}
\newtheorem*{LemmaDef*}{Lemma and Definition}
\newtheorem*{Notation*}{Notation}
\newtheorem*{Problem*}{Problem}
\newtheorem*{Question*}{Question}
\newtheorem*{Remark*}{Remark}
\newtheorem*{Remarks*}{Remarks}
\newtheorem*{Warning*}{Warning}

\title{Real Algebraic Geometry and its Applications}
\author{Tim Netzer}

\begin{document}

\begin{abstract}
This is a survey article on real algebra and geometry, and in particular on its recent applications in optimization and convexity. We first introduce  basic notions and results from the classical theory. We then explain how these relate to optimization, mostly via semidefinite programming. We introduce interesting geometric problems arising from the classification of feasible sets for semidefinite programming. We close with a perspective on the very active area of non-commutative real algebra and geometry.

\bigskip\bigskip
\begin{center}
{\it In memory of Murray Marshall}
\end{center}

\end{abstract}

\maketitle
\section{Real Algebra and Geometry}

The main objects of interest in classical algebraic geometry are varieties, i.e. solution sets of systems of polynomial equations. In order not to complicate this already hard topic, the varieties are often considered over an algebraically closed field. Broadly speaking, {\it real algebraic geometry} deals with real numbers as ground field instead. This involves considering varieties over the reals, but in fact much more. Since the real numbers admit an ordering, one can consider polynomial {\it inequalities}, leading to {\it semialgebraic sets}. It is also important in many proofs not to restrict to real numbers, but allow for general real closed fields (see for example \cite{bcr,ma, netznote, pd} for more detailed explanations):

\begin{Def} A field $R$ is {\it real closed}, if it does not contain a square root of $-1$, and $R(\sqrt{-1})$ is algebraically closed.
\end{Def}

\noindent
Any real closed field has characteristic zero, and the real numbers $\R$ are the standard example of a real closed field. In any real closed field $R$, one obtains a linear ordering by setting $$a\leqslant b \quad :\Leftrightarrow\quad b-a \mbox{ has a square root in } R.$$ This ordering is compatible with the algebraic structure, as it fulfills: $$a\leqslant b \Rightarrow a+c\leqslant b+c\quad \mbox{ and }\quad  0\leqslant a,b \Rightarrow 0\leqslant ab.$$ Any such compatible ordering on a field is called a {\it field ordering}. Now there are many more examples of real closed fields. For example, any ordered field $(F,\leqslant)$ admits an algebraic extension to a real closed field $R$, whose ordering extends the ordering on $F$. This is  called the {\it real closure of $(F,\leqslant)$}. For example, the field $\R(t)$ of rational functions admits a unique field ordering with $r<t$ for all $r\in\R$. Then $t^{-1}$ is positive and smaller than any positive real number. The real closure is thus a real closed field $R$ with infinitesimal elements. 

\begin{Def}
Let $R$ be a real closed field and $p_1,\ldots, p_r\in R[x_1,\ldots, x_n]$ polynomials. The set $$\mathcal W(p_1,\ldots, p_r)=\left\{ a\in R^n\mid p_1(a)\geqslant 0, \ldots, p_r(a)\geqslant 0\right\}$$ is called a {\it basic closed semialgebraic set}. A general {\it semialgebraic set} is a Boolean combination of basic closed semialgebraic sets.
\end{Def}

\noindent An important result on the geometry of semialgebraic sets is the Projection Theorem. It can be proven directly (a non-trivial proof!), but also deduced from a deep model theoretic fact, the so-called {\it quantifier elimination } in real closed fields. 

\begin{Thm}[Projection Theorem] Any polynomial image (for example a projection) of a semialgebraic set is again semialgebraic.
\end{Thm}

\noindent
In classical algebraic geometry, Hilbert's Nullstellensatz provides an algebraic certificate for solvability of a polynomial equation system over an algebraically closed field $K$; the system $$0=p_1(x_1,\ldots, x_n)=\cdots =p_r(x_1,\ldots, x_n)$$ has a solution in $K^n$ if and only if $1\notin (p_1,\ldots, p_r)$, the ideal generated by the equations in the polynomial ring $K[x_1,\ldots, x_n]$. This is a very helpful result, since the last condition can be checked with symbolic computation via Gr\"obner bases.

\noindent
The set $\mathcal W(p_1,\ldots, p_r)$ is the solution set of the system of polynomial {\it inequalities} $$p_1(x_1,\ldots, x_n)\geqslant 0, \ldots, p_r(x_1,\ldots, x_n)\geqslant 0$$ in $R^n$. One of the fundamental results in real algebra provides a similar characterization for solvability of this system. For this {\it Nichtnegativstellensatz} we need the notion of a preordering, replacing the ideal in Hilbert's Nullstellensatz:

\begin{Def}
Let $A$ be a commutative ring and $p_1,\ldots, p_r\in A$. The {\it preordering} $$\mathcal P(p_1,\ldots, p_r)$$ generated by $p_1,\ldots, p_r$ is the smallest set closed under addition and multiplication, containing $p_1,\ldots, p_r$ and all sums of squares. In closed form: $$\mathcal P(p_1,\ldots, p_r)=\left\{ \sum_{e\in\{0,1\}^r} \sigma_e\cdot p_1^{e_1}\cdots p_r^{e_r}\mid \sigma_e \mbox{ sum of squares in } A \right\}.$$ \end{Def}

\noindent
In the case of a polynomial ring $A=R[x_1,\ldots, x_n],$ the preordering $\mathcal P(p_1,\ldots, p_r)$ contains polynomials that are obviously nonnegative as functions on $\mathcal W(p_1,\ldots, p_r)$. A full characterization of nonnegative functions is the following:

\begin{Thm}[Nichtnegativstellensatz] Let $R$ be a real closed field and $p,p_1,\ldots, p_r\in R[x_1,\ldots, x_n]$. Then the following are equivalent: \begin{itemize} \item[(i)] $p\geqslant 0 \mbox{ on } \mathcal W(p_1,\ldots, p_r)$
\item[(ii)] $fp=p^{2e} + g$ for some $f,g\in \mathcal P(p_1,\ldots, p_r), f\neq 0, e\in\N.$ 
\end{itemize} In particular $\mathcal W(p_1,\ldots, p_r)=\emptyset$ if and only if $-1\in\mathcal P(p_1,\ldots, p_r).$
\end{Thm}

\noindent
The Nichtnegativstellensatz deserves some detailed remarks. First, the case $r=0$ corresponds to $\mathcal W(p_1,\ldots, p_r)=R^n$ and $\mathcal P(p_1,\ldots, p_r)$ the set of sums of squares of polynomials. In this case the statement simplifies to $$p\geqslant 0 \mbox{ on } R^n \quad\Leftrightarrow\quad q^2p=q_1^2+\cdots +q_s^2$$ for some $0\neq q,q_1,\ldots, q_s\in R[x_1,\ldots, x_n].$ In words: \begin{center}{\it Every globally nonnegative polynomial is a sum of squares of rational functions.} \end{center} This is precisely {\it Hilbert's 17th Problem}, solved by Artin in 1926 \cite{artin}. It is {\it not} possible to get rid of the denominator $q$ in the result. The {\it Motzkin polynomial} \cite{motz}$$x_1^4x_2^2+x_1^2x_2^4-3x_1^2x_2^2 +1$$ is nonnegative on $\R^2$, but not a sum of squares in $\R[x_1,x_2].$ It is even more surprising that this example was found only in 1967, since we know today that hardly any nonnegative polynomial is a sum of squares \cite{blekh}. For more technical and historical remarks on Hilbert's 17th Problem and the Nichtnegativstellensatz see \cite{pd}. 

\noindent
Second, since $-1\geqslant 0$ on $W$ if and only if $W=\emptyset$, solvability of the inequality system $$p_1(x_1,\ldots, x_n)\geqslant 0,\ldots, p_r(x_1,\ldots, x_n)\geqslant 0$$ is characterized by the condition $-1\notin \mathcal W(p_1,\ldots, p_r)$. Interestingly, this algebraic condition also admits an effective algorithmic approach, which is numerical however, in contrast to the symbolic approach to Hilbert's Nullstellensatz. We give some more detailed explanations in the next section.

\noindent
Third, we want to give an idea of the proof of Hilbert's 17th Problem. The proof of the general Nichtnegativstellensatz ist slightly more involved, but not conceptually different. One direction is clear; if $q^2p$ is a sum of squares, then $q^2p$ is globally nonnegative, and so is $p$ by continuity, since $q$ vanishes only on a low-dimensional set. The other direction splits up into two parts (Theorem \ref{mth} and Theorem \ref{int} below), and is far more complicated than the proof of  Hilbert's Nullstellensatz. Especially the first part relies again on hard model-theoretic facts. It translates geometric positivity of a polynomial to an abstract positivity in the field of rational functions:
\begin{Thm}\label{mth}
Let $R$ be a real closed field and $p\in R[x_1,\ldots, x_n]$ with $p\geqslant 0$ on $R^n$. Then for any field ordering $\geq$ of $R(x_1,\ldots, x_n)$ we have $p\geq 0$.
\end{Thm}
\begin{proof} Assume $\geq$ is a field ordering of $R(x_1,\ldots, x_n)$ with $p<0$. Let $S$ be the real closure of $R(x_1,\ldots, x_n)$ with respect to this ordering. Then  the following semialgebraic set in $S^n$ is non-empty: $$\left\{ a\in S^n\mid p(a)<0\right\}.$$ In fact the tuple of variables $(x_1,\ldots, x_n)$, which are all elements in $R(x_1,\ldots, x_n)$ and thus in $S$, belongs to this set. Now by Tarski's Transfer Principle (see Theorem \ref{tar} below) the set contains a point from $R^n$, contradicting the fact that $p\geqslant 0$ on $R^n$.
\end{proof}

\noindent
The model theory is contained in the following result, which can again be deduced from the even stronger quantifier elimination mentioned above.

\begin{Thm}[Tarski's Transfer Principle]\label{tar}
Let $S/R$ be an extension of real closed fields. If a nonempty semialgebraic set in $S^n$ is defined by polynomials over $R$, then it contains a point from $R^n$.\end{Thm}

\noindent
The second part of the proof of the Nichtnegativstellensatz is easier to prove and of algebraic nature, see \cite{bcr, ma, netznote, pd}:

\begin{Thm}\label{int}
Let $K$ be a field and $p\in K,$ such that in each field ordering $\geq$ of $K$ we  have $p\geq 0$. Then $p$ is a sum of squares in $K$. \end{Thm}

\noindent
The Nichtnegativstellensatz yields {\it denominators} in the algebraic certificate, i.e. we have to multiply $p$ with some $f$ before we obtain a representation in $\mathcal P(p_1,\ldots, p_r).$ The first denominator-free result is Schm\"udgen's Theorem, which triggered a whole series of new developments.

\begin{Thm}[Schm\"udgen \cite{schm4}] Let $p,p_1,\ldots, p_r\in\R[x_1,\ldots, x_n]$ be such that $$\mathcal W(p_1,\ldots,p_r)\subseteq \R^n$$ is bounded. Then $p>0 \mbox{ on } \mathcal W(p_1,\ldots, p_r)$ implies $p\in \mathcal P(p_1,\ldots, p_r).$
\end{Thm}

\noindent Let us add some comments on Schm\"udgen's Theorem. First, it only holds for $\mathbb R$, not for arbitrary real closed fields.  Second, the boundedness of $\mathcal W(p_1,\ldots,p_r)$ is a necessary condition, as is the strict positivity of $p$ in general. Third, the result admits innovative applications to polynomial optimization, as we will demonstrate in the next section.

\noindent
We conclude with some remarks on the question whether $$\forall p \qquad p\geqslant 0 \mbox{ on } \mathcal W(p_1,\ldots, p_r) \Rightarrow p\in \mathcal P(p_1,\ldots, p_r)$$ can ever hold. 
If $n=1$, this is quite frequent. It holds true whenever the defining polynomials $p_i$ are chosen in the {\it canonical way} for the definition of $\mathcal W(p_1,\ldots,p_r)\subseteq\R$ (see \cite{km,kms}). Surprisingly, if $n\geqslant3$ and $\mathcal W(p_1,\ldots, p_r)$ has nonempty interior, it {\it never} holds. There are always nonnegative polynomials $p$ that do not belong to $\mathcal P(p_1,\ldots,p_r),$ no matter which and how many $p_i$ we choose to define the set \cite{sch1}. For $n=2$ the situation is quite subtle. For certain compact sets there is an affirmative answer by deep results of Scheiderer \cite{sch4}, and there is an interesting non-compact example by Marshall \cite{strip}.

\section{Optimization}\label{opt}
The results from the last section are closely related to optimization, mostly via semidefinite programming. 

\begin{Def}
A {\it semidefinite program} is an optimization problem of the following form:
\begin{align*}
\mbox{ minimize } &\quad c_1a_1+\cdots +c_na_n \\ 
\mbox{ subject to } &\quad  M_0+a_1M_1+ \cdots + a_nM_n \succeq 0,
\end{align*}
  where $c_1,\ldots, c_n\in\R,$ $M_0,M_1,\ldots, M_n\in {\rm Sym}_N(\R)$ are symmetric matrices, and $M\succeq 0$ means that $M$ is positive semidefinite.  
\end{Def}

\noindent So the feasible set of a semidefinite program is an affine-linear section of a cone of positive semidefinite matrices. Semidefinite programming is a generalization of linear programming. The feasible set $$\left\{ a\in\R^n \mid M_0+a_1M_1+\cdots +a_nM_n\succeq 0\right\}$$ is a  polyhedron if all matrices are diagonal. With non-diagonal matrices we obtain a larger class of sets.  For example, the condition $$\left(\begin{array}{cc}1 & 0 \\0 & 1\end{array}\right)  +a_1 \left(\begin{array}{cc}1 & 0 \\0 & -1\end{array}\right) +a_2 \left(\begin{array}{cc}0 & 1 \\1 & 0\end{array}\right)\succeq 0 $$ defines the unit disk in $\R^2$.  Solving a semidefinite program can be done with efficient numerical algorithms, mostly interior-point methods, and there is also a duality theory for semidefinite programs (see for example \cite{hbsdp}).

\noindent
The connection to polynomials and sums of squares is via {\it Gram matrices}. Let $d\in\N$ be a fixed degree, and $$\mathfrak m_d=(1,x_1,\ldots, x_n, x_1^2, x_1x_2,\ldots,x_1^d,\ldots , x_n^d)$$ be the vector of all monomials of degree $\leqslant d$. If $N$ denotes the size of $\mathfrak m_d$, then for any symmetric $N\times N$ matrix $M\in{\rm Sym}_N(\R)$ we obtain a polynomial $$p_M=\mathfrak m_d M \mathfrak m_d^t\in \R[x_1,\ldots, x_n]$$ of degree $\leqslant 2d$, and  any polynomial of degree $\leqslant 2d$ is of this form.

\begin{Def}
Any $M\in {\rm Sym}_N(\R)$ with $p_M=p$ is called a {\it Gram matrix } of $p$.
\end{Def}

\noindent
The connection between sums of squares and semidefinite programming relies essentially on the following observation:

\begin{Lemma}\label{sos} A polynomial $p\in\R[x_1,\ldots, x_n]$ with $\deg(p)\leqslant 2d$ is a sum of squares in $\R[x_1,\ldots, x_n]$ if and only if $p$ has a positive semidefinite Gram matrix of size $N$.
\end{Lemma} 
\begin{proof}
"$\Rightarrow$": Let $p=p_1^2+\cdots + p_s^2$ with polynomials $p_i$. It is easy to see that $\deg(p)\leqslant 2d$ implies $\deg(p_i)\leqslant d$ for all $i=1,\ldots, r$ (highest degree parts in squares are squares, and cannot cancel additively). So there are (column) vectors $c_i\in\R^N$ with $p_i=\mathfrak m_d c_i$. Then $$p=\sum_i p_i^2 =\sum_i (\mathfrak m_d c_i)(\mathfrak m_d c_i)^t= \mathfrak m_d\left( \sum_i c_ic_i^t \right)\mathfrak m_d^t,$$ and thus  $M=\sum_i c_ic_i^t$ is a positive semidefinite Gram matrix of $p$.

"$\Leftarrow$": Write $p=\mathfrak m_d M\mathfrak m_d^t$ for some positive semidefinite $M\in{\rm Sym}_N(\R)$. Every positive semidefinite matrix is a sum of rank one squares, i.e. there are vectors $c_i\in\R^N$ with $M=\sum_i c_ic_i^t$. Now $$p=\mathfrak m_dM\mathfrak m_d^t= \mathfrak m_d \left( \sum_i c_ic_i^t\right)\mathfrak m_d^t = \sum_i (\mathfrak m_d c_i)(\mathfrak m_dc_i)^t=\sum_i (\mathfrak m_dc_i)^2$$ is a sum of squares.
\end{proof}

\noindent
This observation is the key ingredient in Lasserre's hierarchy for polynomial optimization \cite{las2}. Given $p,p_1,\ldots, p_r\in \R[x_1,\ldots, x_n]$, the initial problem is to determine \begin{align*} \inf &\quad p(a) \\ \mbox{ s.t.} &\quad  a\in \mathcal W(p_1,\ldots,p_r).\end{align*} We will denote this problem by $(P)$, and its optimal value by $p^*$. It is a general constrained polynomial optimization problem, and thus hard to solve. In particular, there is no convexity or linearity involved. The idea now is to relax this problem to a series of easier ones. For fixed $d\in \N$ we consider the following problem, which we denote by $(P_d)$: \begin{align*} \qquad 
\sup \quad&\quad \lambda \\ \mbox{s.t.} \quad &\quad p-\lambda = \sum_{e\in \{0,1\}^r} \sigma_e p_1^{e_1}\cdots p_r^{e_r}, \quad  \sigma_e \mbox{ sums of squares of degree} \leqslant 2d.
\end{align*}
So we maximize $\lambda,$ such that  $p-\lambda$ admits a representation in the preordering $\mathcal P(p_1,\ldots, p_r)$, with a bound of $2d$ on the degree of the sums of squares $\sigma_e.$ The optimal value of $(P_d)$ is denoted by $\lambda_d^*$.

\begin{Thm}[Lasserre] With $p,p_1,\ldots,p_r\in \R[x_1,\ldots, x_n]$ as above we have:
\begin{itemize}
\item[(i)] Each $(P_d)$ is a semidefinite program.

\item[(ii)] The sequence $\left( \lambda^*_d\right)_{d\in\N}$ is monotonically increasing, with $\lambda_d^*\leqslant p^*$ for all $d$.  
\item[(iii)] If $\mathcal W(p_1,\ldots, p_r)$ is bounded, then $\lim_d \lambda_d^* =p^*.$
\end{itemize}
\end{Thm}
\begin{proof}
(i) For two polynomials $p,q$ we write $p \sim q$ if $p$ and $q$ coincide up to the constant term. Now consider the following set: $$S= \left\{ (M_e)_{e\in\{0,1\}^r}\in {\rm Sym}_N(\R)^{2^r}\mid \forall e\ M_e\succeq 0,  \sum_e \mathfrak m_dM_e\mathfrak m_d^t \cdot p_1^{e_1}\cdots p_r^{e_r}\sim p\right\}.$$ It is not hard to see that $ S$ can be realized as an affine-linear section of the convex cone of all positive semidefinite matrices of size $2rN$. This involves building a large block-diagonal matrix from the matrices $M_e$, and comparing coefficients (except for the constant term) in the equation  $$\sum_e \mathfrak m_dM_e\mathfrak m_d^t \cdot p_1^{e_1}\cdots p_r^{e_r}=p.$$ Thus $S$ is the feasible set of a semidefinite program. Now $(P_d)$ just means minimizing the constant term in $\sum_e \mathfrak m_dM_e\mathfrak m_d^t \cdot p_1^{e_1}\cdots p_r^{e_r},$ which is linear in the entries of the matrices $M_e$. This uses Lemma \ref{sos}, i.e. the fact that each sum of squares $\sigma_e$ of degree $\leqslant 2d$ is of the form $\mathfrak m_d M_e\mathfrak m_d^t$ for some positive semidefinite matrix $M_e$. So $(P_d)$ is a semidefinite program.

(ii) It is clear that the values $\lambda_d^*$ increase with $d$. Now assume $p-\lambda$ has a representation as desired in $(P_d)$. Then $p-\lambda$ is obviously nonnegative on $\mathcal W(p_1,\ldots,p_r)$, since it belongs to $\mathcal P(p_1,\ldots, p_r)$. Thus $\lambda\leqslant p^*.$

(iii) For any $\epsilon >0$ we have $p-p^*+\epsilon >0$ on $\mathcal W(p_1,\ldots, p_r)$. By Schm\"udgen's Theorem thus $ p-p^*+\epsilon \in \mathcal P(p_1,\ldots, p_r)$. In such a fixed representation there is clearly an upper bound $2d$ on the degrees of the sums of squares $\sigma_e$, and thus $\lambda_d^* \geqslant p^*-\epsilon.$
\end{proof}

\noindent
This relaxation method for polynomial optimization is implemented in the free Matlab plugin {\it Yalmip} \cite{lof}. It works well in practice if the degree and the dimension of the involved polynomials is not too large.  The rate of convergence is closely linked to degree bounds in Schm\"udgen's Theorem, which are analyzed in \cite{pd,schw1,schw2}.

\noindent
Beyond being useful for polynomial optimization, semidefinite programming also raises some interesting geometric questions, that we will describe in the following section.

\section{Algebraic Convexity}

The feasible sets of semidefinite programming turn out to be of interesting geometric nature. They are called {\it spectrahedra} \cite{rago}:

\begin{Def}
A set $S\subseteq \R^n$ is called a {\it spectrahedron}, if there exist  symmetric matrices $M_0,\ldots,M_n$    such that $$ S=\left\{ a\in\R^n\mid M_0+a_1M_1+\cdots +a_nM_n\succeq 0\right\}.$$ Recall that $M\succeq 0$ means that $M$ is positive semidefinite. The expression $$M_0+x_1M_1+\cdots +x_nM_n$$ is called a {\it linear matrix polynomial}, and the expression $$M_0+x_1M_1+\cdots +x_nM_n\succeq 0$$ a {\it linear matrix inequality}.
\end{Def}

\noindent
It is straightforward to see that spectrahedra are closed, convex and even basic closed semialgebraic. The principal minors of $M_0+x_1M_1+\cdots +x_nM_n$ for example define $S$ as a basic closed semialgebraic set.

\begin{Example}
The convex hull of two disjoint  disks in the plane  is a closed, convex and semialgebraic set. It is however not basic closed semialgebraic, i.e. not definable by simultaneous polynomial inequalities. This is a nice exercise, see also \cite{sinn}. So it is not a spectrahedron. This set is called the {\it football stadium}.

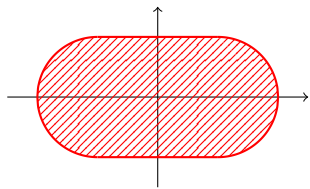
\begin{figure}[ht]
\begin{center} 
\begin{tikzpicture}[scale=0.8]
\draw[->] (-1.5,0) -- (3.5,0) node[right]{};
\draw[->] (1,-1.5) -- (1,1.5) node[above]{};
\pgfsetstrokecolor{red};
\pgfsetfillpattern{north east lines}{red};
\filldraw[smooth,domain=1:0 ,samples=\no, thick] plot({-\x},{sqrt(1- \x^2)});
\filldraw[smooth,domain=1:0 ,samples=\no, thick] plot({-\x},{0-sqrt(1- \x^2)});
\filldraw[smooth,domain=1:0 ,samples=\no, thick] plot({\x+2},{sqrt(1- \x^2)});
\filldraw[smooth,domain=1:0 ,samples=\no, thick] plot({\x+2},{0-sqrt(1- \x^2)});
\fill (-1,0) -- (0,-1) -- (0,1) -- (-1,0) -- cycle; 
\fill (3,0) -- (2,-1) -- (2,1) -- (3,0) -- cycle;
\fill (0,1) -- (2,1) -- (2,-1) -- (0,-1) -- (0,1) -- cycle; 
\draw[-, thick] (0,1) -- (2,1) node[right]{};
\draw[-, thick] (0,-1) -- (2,-1) node[right]{};
\end{tikzpicture}
\end{center}
\end{figure}

\end{Example}

\noindent
But spectrahedra have more properties. For example, each face of a spectrahedron $S$ is exposed, i.e. realizable as the intersection of $S$ with a supporting hyperplane.

\begin{Example}
Consider the set $\left\{ (a,b)\in\R^2\mid a^3\leqslant b, -1\leqslant a,  0\leqslant b\leqslant 1\right\}.$ It is compact,  basic closed semialgebraic and convex, but has a non-exposed extreme point (the origin). It is thus not a spectrahedron.
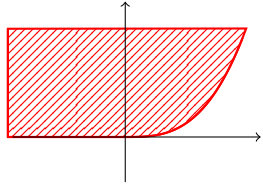
\begin{figure}[ht]
 \begin{center}
\begin{tikzpicture}[scale=0.6]
\begin{scope}
\pgfsetstrokecolor{red};
\pgfsetfillpattern{north east lines}{red};
\draw[smooth,domain=0:2.6, thick] plot({\x},{(0.5*\x)^3});
\draw[thick] (-2.6,0) -- (-2.6,2.4);
\draw[thick] (-2.6,0)--(0,0);
\filldraw[smooth,domain=0:2.68, thick] plot({\x},{(0.5*\x)^3}) -- (-2.6,2.4) -- (-2.6,0) -- (0,0);
\end{scope}
\draw[->] (-2.5,0) -- (3,0) node[right]{};
\draw[->] (0,-1) -- (0,3) node[above]{};
\end{tikzpicture}
\end{center}
 \end{figure}

\end{Example}

\noindent
But these properties by far not  characterize spectrahedra. The crucial property is {\it hyperbolicity}, which in fact implies all the before mentioned properties.

\begin{Def}
Let $p\in \R[x_1,\ldots, x_n]$ and $e\in \R^n$. 

\noindent
(i) $p$ is called {\it hyperbolic with respect to $e$,} if $p(e)\neq 0$ and for each $v\in \R^n$, the univariate polynomial $$p_v(t):= p(e+tv)\in \R[t]$$ has only real roots.

\noindent
(ii) If $p$ is hyperbolic with respect to $e$, then  $$\mathcal H_e(p)=\left\{ a\in \R^n \mid \forall \lambda \in (0,1] : \ p(\lambda e+(1-\lambda)a) \neq 0 \right\}$$ is called the {\it hyperbolicity region of $h$ with respect to $e$}.
\end{Def}

\noindent
Geometrically, a polynomial $p$ is hyperbolic if any real line through $e$ intersects the complex hypersurface  of $p$ in only real points. The hyperbolicity region is the area within the innermost ring or zeroes of $p$ around $e$.  Interestingly, it can be shown that hyperbolicity regions are always convex and basic closed semialgebraic. They also have only exposed faces (see \cite{ren}).  
\noindent
Note that a hyperbolic polynomial is sometimes also called a {\it real zero polynomial} in the literature, and the hyperbolicity region a {\it rigidly convex set}; the notion of hyperbolicity is then used for a similar concept for 
homogeneous polynomials.

\begin{Example}(i) The polynomial $p=1-x_1^2-x_2^2\in\R[x_1,x_2]$ is hyperbolic w.r.t.\ $e=(0,0)$. The hyperbolicity region is the unit disk.
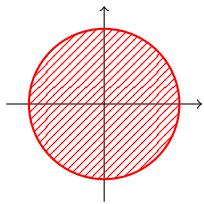
\begin{figure}[ht]
\begin{center}
\begin{tikzpicture}
\draw[->] (-1.3,0) -- (1.3,0) node[right]{};
\draw[->] (0,-1.3) -- (0,1.3) node[right]{};
\pgfsetstrokecolor{red};
\pgfsetfillpattern{north east lines}{red};
\filldraw[thick] (0,0) circle (1cm);
\end{tikzpicture}
\end{center}
\end{figure}

\noindent
(ii) The polynomial $p=1-x_1^4-x_2^4$ is not hyperbolic w.r.t.\ $e=(0,0)$ (or any other point). On any line through the origin, the quartic $p_v(t)\in\R[t]$ has $2$ real and $2$ strictly complex roots. It is thus not hard to see that the region $\left\{ (a,b)\in\R^2\mid a^4+b^4\leqslant 1\right\}$ is not the hyperbolicity region of any hyperbolic polynomial. This set is called the {\it TV-screen}.
\begin{figure}[!htbp]
\begin{center}
\begin{tikzpicture}
\draw[->] (-1.3,0) -- (1.3,0) node[right]{};
\draw[->] (0,-1.3) -- (0,1.3) node[right]{};
\pgfsetstrokecolor{red};
\pgfsetfillpattern{north east lines}{red};
\filldraw[smooth,domain=1:0 ,samples=\no, thick] plot({\x},{ sqrt(sqrt(1- \x^4))});
\filldraw[smooth,domain=1:0 ,samples=\no, thick] plot({-\x},{ sqrt(sqrt(1- \x^4))});
\filldraw[smooth,domain=1:0,samples=\no, thick] plot({\x},{0 -(sqrt(sqrt(1- \x^4)))});
\filldraw[smooth,domain=1:0,samples=\no, thick] plot({(-\x)},{0 -(sqrt(sqrt(1- \x^4)))});
\fill (1,0) -- (0,1) -- (-1,0) -- (0,-1) -- cycle; 
\end{tikzpicture}
\end{center}
\end{figure}

\end{Example}

\begin{Prop}
Every spectrahedron with nonempty interior is the hyperbolicity region of a hyperbolic polynomial.
\end{Prop}
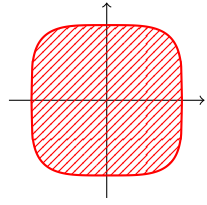
\begin{proof}[Sketch of proof] Assume without loss of generality that the origin belongs to the interior of the spectrahedron $S$. Then $S$ can be defined by a {\it monic} linear matrix inequality, i.e. $$S=\left\{ a\in \R^n\mid I+a_1M_1+\cdots +a_nM_n\succeq 0\right\}.$$ This involves some technical details that we skip. Then $$p=\det(I+x_1M_1+\cdots +x_nM_n)\in\R[x_1,\ldots, x_n]$$ is hyperbolic with respect to the origin. This follows easily from the fact that symmetric matrices have only real Eigenvalues. It is then also not hard to see that the hyperbolicity region of $p$ coincides with $S$.
\end{proof}

\noindent
So the TV-screen is not a spectrahedron, although is is convex, basic closed semialgebraic and has only exposed faces. 
One of the main open questions concerning spectrahedra is the following. If true, it would classify spectrahedra in terms of the behavior of their boundary surface. 

\begin{Conjecture}[Geometric Lax Conjecture]
Every hyperbolicity region is a spectrahedron.
\end{Conjecture}

\noindent
In full generality, the conjecture is open. There are different approaches and partial positive results (for example \cite{br, cliff}), but most importantly, a solution in dimension two. For simplicity, we will assume from now on that $e$ is the origin and $p(e)=1$.

\begin{Thm}[Helton \& Vinnikov \cite{hevi}] The Geometric Lax Conjecture holds true in $\R^2.$ Even stronger, ever hyperbolic polynomial $p\in\R[x_1,x_2]$ has a monic determinantal representation $$p=\det\left(I+x_1M_1+\cdots +x_nM_n \right)$$ with symmetric matrices $M_i$.
\end{Thm}

\noindent
This is a deep mathematical result, and the proof employs hard algebraic geometry. There are now some easier and also algorithmic proofs of slightly weaker statements (see \cite{pl1,pl2}). 

\noindent
Concerning determinantal representations, let us mention two more results.

\begin{Thm}[Kummer \cite{kummer}]
For every hyperbolic polynomial $p\in \R[x_1,\ldots, x_n]$ there is a determinantal representation of some multiple $$qp=\det(I+x_1M_1+\cdots+x_nM_n).$$
\end{Thm}

\noindent
Unfortunately, there is no control over the factor $q$ in the representation. So the hyperbolicity region of $qp$ might be strictly smaller than the one of $p$. The next result is a statement about rational representations (with no obvious consequences for the Geometric Lax Conjecture).

\begin{Thm}[Netzer, Plaumann \& Thom \cite{neplth}] For every hyperbolic polynomial $p\in \R[x_1,\ldots, x_n]$ there is a symmetric matrix $M$ of homogeneous, rational, degree one functions, with $p=\det(I+M).$
\end{Thm}

Now passing from spectrahedra to their linear images increases the class of sets a lot.

\begin{Def}
The linear image of a spectrahedron is called a {\it spectrahedral shadow}.
\end{Def}

\noindent
The class of spectrahedral shadows is closed under any reasonable operation on convex sets. This includes duals, closures, interiors, products, sums and convex hulls of unions  (see for example \cite{habil}). 
 Spectrahedral shadows are convex and semialgebraic (by the Projection Theorem), but {\it no other} necessary condition is known:
 
 \begin{Conjecture}[Helton-Nie Conjecture] Every convex semialgebraic set is a spectrahedral shadow.
 \end{Conjecture}

\noindent
If this was true, it would allow to apply semidefinite programming on any convex semialgebraic set. One can pull back the problem from the linear image of a spectrahedron to the spectrahedron itself.
There are many results in support of the Helton-Nie Conjecture. The basic construction of  spectrahedral shadows is the following, building a bridge to results of real algebra, in particular Positivstellens\"atze:

\begin{Thm}[Lasserre \cite{las}]
Let   $p_1,\ldots, p_r\in\R[x_1,\ldots, x_n]$ and set $$W=\mathcal W(p_1,\ldots, p_r)\subseteq \R^n.$$ Assume there exists some $d\in\N$, such that whenever a polynomial $\ell$ of degree $\leqslant 1$ fulfills $\ell\geqslant 0$ on $W$, then $$\ell=\sum_e \sigma_e p_1^{e_1}\cdots p_r^{e_r}\in\mathcal P(p_1,\ldots, p_r)$$ with sums of squares $\sigma_e$ of degree $\leqslant 2d$. Then the closed convex hull $$\overline{{\rm conv}(W)}$$ is a spectrahedral shadow.
\end{Thm}
\begin{proof}[Sketch of proof] The polar dual $W^\circ= \left\{ \ell=\ell_0+\ell_1x_1+\cdots +\ell_nx_n\mid \ell_i\in\R,  \ell \geqslant 0 \mbox{ on } W\right\}$ is a spectrahedral shadow. This can be seen via Gram matrices as above, since every $\ell\in W^\circ$ admits a preordering representation with degree bounds. The dual of a spectrahedral shadow is again a spectrahedral shadow, and thus so is the double dual of $W,$ which coincides with $\overline{{\rm conv}(W)}$.
\end{proof}

\noindent
So if preordering representations of nonnegative linear polynomials (with degree bounds) can be proven for a set, its closed convex hull is a spectrahedral shadow. Helton and Nie prove this for a large class of sets \cite{heni,heni2}. For example: 

\begin{Thm}[Helton \& Nie]
Assume $p_1,\ldots, p_r\in \R[x_1,\ldots, x_n]$ are such that $$W=\mathcal W(p_1,\ldots,p_r)\subseteq\R^n$$ is convex and bounded. Further assume the negative Hessian matrices $-\mathbf H(p_1), \ldots, -\mathbf H(p_r)$ are all sums of Hermitian squares in the matrix ring ${\rm Mat}_n\left(\R[x_1,\ldots, x_n]\right)$. Then $W$ is a spectrahedral shadow.
\end{Thm}

\begin{Example}
Consider the TV-screen $W=\mathcal W(1-x_1^4-x_2^4)\subseteq \R^2,$ which is not a spectrahedron. Compute $$-\mathbf H(1-x_1^4-x_2^4)= \left(\begin{array}{cc}12x_1^2 & 0 \\0 & 12x_2^2\end{array}\right)=\left(\begin{array}{cc}\sqrt{12}x_1 & 0 \\0 & \sqrt{12}x_2\end{array}\right)^t\cdot \left(\begin{array}{cc}\sqrt{12}x_1 & 0 \\0 & \sqrt{12}x_2\end{array}\right).$$ Thus the TV-screen is a spectrahedral shadow.
\end{Example}

\noindent
Recently, Scheiderer has settled the Helton-Nie Conjecture in dimension $2$, building upon his deep results about sums of squares on algebraic curves:

\begin{Thm}[Scheiderer \cite{sch2}]
Every convex semialgebraic set in $\R^2$ is a spectrahedral shadow.
\end{Thm} 

\noindent
For a more thorough treatment of the topics in the last two section see for example also \cite{par}.

\section{Non-commutative theory}

In recent years, the theory of non-commutative real geometry has attracted more and more interest, in part motivated by applications in systems engineering and control theory \cite{he1}. The most important algebraic objects are non-commutative polynomials. The non-commutative polynomial ring $$\R\langle z_1,\ldots,z_n\rangle $$ has as its elements  $\R$-linear combination of words in the letters $z_1,\ldots, z_n$, which do not commute. So $z_1z_2$ and $z_2z_1$ are different polynomials. Non-commutative polynomials  are naturally evaluated at tuples of matrices; if $A_1,\ldots,A_n\in{\rm Mat}_s(\R)$, then $p(A_1,\ldots,A_n)\in {\rm Mat}_s(\R)$.
There is an involution on $\R\langle z_1,\ldots, z_n\rangle$ with $z_i^*=z_i$ for all $i$ and $*={\rm id}$ on $\R$. Thus $*$ just reverses the order of variables in a monomial, for example $$(7z_1+z_1z_2)^*=7z_1+z_2z_1.$$ Let $\R\langle z_1,\ldots,z_n\rangle_h$ denote the set of Hermitian polynomials, i.e. fixed points of the involution. For any $p\in\R\langle z_1,\ldots, z_n\rangle_h$ and $A_1,\ldots,A_n\in {\rm Sym}_s(\R)$, the matrix $p(A_1,\ldots,A_n)$ is again symmetric. So for $s\geqslant 1$ and $p_1,\ldots,p_r\in\R\langle z_1,\ldots,z_n\rangle_h$ it makes sense to define $$\mathcal W_s(p_1,\ldots,p_r)=\left\{ (A_1,\ldots,A_n)\in{\rm Sym}_s(\R)^n\mid \forall i \quad p_i(A_1,\ldots,A_n)\succeq 0\right\}$$ and $$\mathcal W(p_1,\ldots,p_r)=\bigcup_{s\geqslant 1} \mathcal W_s(p_1,\ldots, p_r).$$ So such a {\it non-commutative semialgebraic set} consists of a collection of matrix tuples, for all matrix sizes simultaneously. Since it contains much more information than a classical semialgebraic set in $\R^n$, it is not surprising that stronger Positivstellens\"atze can  be proven. Note that $\Sigma^2\R\langle z_1,\ldots, z_n\rangle$ here denotes the set of sums of {\it Hermitian squares}, i.e. sums of elements of the form $p^*p$ with $p\in\R\langle z_1,\ldots, z_n\rangle.$ This is the correct notion to reflect positivity. The following is a non-commutative version of Hilbert's 17th problem, {\it without denominators}:

\begin{Thm}[Helton \cite{he0}] Assume $p\in\R\langle z_1,\ldots, z_n\rangle_h$ fulfills $$p(A_1,\ldots,A_n)\succeq 0$$ for all $A_1,\ldots,A_n\in {\rm Sym}_s(\R)$ and $s\geqslant 1$. Then $p\in\Sigma^2\R\langle z_1,\ldots, z_n\rangle.$
\end{Thm}

\noindent
Here, as in most of the non-commutative results, the proof methods differ quite strongly from the commutative theory. They are much more functional-analytic in nature, and often require some knowledge about operator theory; see also \cite{schm} for more details.
 There are more  Positivstellens\"atze in the spirit of the above, which we don't mention here, see for example \cite{aleknetzthom,he2,he5}.
It is unclear to which extend a Projection Theorem (or quantifier elimination) holds for non-commutative semialgebraic geometry. It is definitely even harder as in the commutative setup, as some examples indicate. 

There is also an interesting notion of convexity in the non-commutative setup (see also \cite{he1}). For $A=(A_1,\ldots,A_n)\in {\rm Sym}_s(\R)^n, B=(B_1,\ldots,B_n)\in {\rm Sym}_r(\R)^n$ and $V\in{\rm Mat}_{s,r}(\R)$ we set $$A\oplus B= \left(\left(\begin{array}{cc}A_1 & 0 \\0 & B_1\end{array}\right),\ldots, \left(\begin{array}{cc}A_n & 0 \\0 & B_n\end{array}\right) \right)\in {\rm Sym}_{s+r}(\R)^n$$ and $$V^tAV=(V^tA_1V,\ldots, V^tA_nV)\in {\rm Sym}_r(\R)^n.$$
\begin{Def}\label{freeconv}
Let $W_s\subseteq{\rm Sym}_s(\R)^n$ be given, for every $s\geqslant 1.$ 
The collection $W=\bigcup_{s\geqslant 1} W_s$ is {\it matrix convex}, if it fulfills the following conditions:
\begin{itemize}
\item[(i)] $A\in W_s, B\in W_t \Rightarrow A\oplus B\in W_{s+t}$ \item[(ii)] $A\in W_s$, $V\in{\rm Mat}_{s,r}(\R), V^tV=I_r$ $\Rightarrow V^tAV\in W_r$.
\end{itemize}
\end{Def}

\noindent
Matrix convexity implies classical convexity for each $W_s$, but is  stronger than this in general. The standard example of a matrix convex set is the following. Let $M_1,\ldots, M_n\in{\rm Sym}_d(\R)$ be given and define $$W_s=\left\{ (A_1,\ldots, A_n)\in {\rm Sym}_s(\R)^n\mid I_s\otimes I_d+A_1\otimes M_1+\cdots +A_n\otimes M_n \succeq 0\right\}.$$ Here $\otimes$ denotes the Kronecker product of matrices. Then the collection $W=\bigcup_{s\geqslant 1} W_s$ is matrix convex. Such a set is also called a {\it non-commutative spectrahedron}. Helton \& McCullough \cite{he4}  prove (under weak additional assumptions)  that every matrix convex non-commutative semialgebraic set is of this form.

\noindent
Clearly intersections of matrix convex sets are matrix convex, and there is thus the notion of the {\it matrix convex hull} of a set $W=\bigcup_{s\geqslant 1} W_s.$ If $W$ already fulfills condition (i) from Definition \ref{freeconv} (as is the case for semialgebraic sets $\mathcal W(p_1,\ldots,p_r)$), then the matrix convex hull is obtained by just adding all compressions $V^tAV$ as in (ii). The matrix convex hull can however behave very badly, even at scalar level:

\begin{Thm}[Alekseev, Netzer \& Thom \cite{aleknetzthom}]
There are $p_1,\ldots,p_r\in\R\langle z_1,\ldots,z_n\rangle_h,$ such that $$\left\{ v^tAv \mid s\geqslant 1, A\in \mathcal W_s(p_1,\ldots,p_r),  v\in \R^s, v^tv=1\right\}\subseteq \R^n$$ is  not semialgebraic (in the classical sense). \end{Thm}

To finish this review article, we mention a rather surprising application of non-commutative real algebra to group theory. For this let $G$ be a group and $\R[G]$ its group algebra, i.e. its elements are formal (finite) linear combinations of group elements, with multiplication induced by the group operation. There is an involution on $\R[G]$, defined by $g^*:=g^{-1}$ for $g\in G$, and $*={\rm id}$ on $\R$. Then $\Sigma^2\R[G]$ denotes the set of sums of Hermitian squares.

\noindent If $S\subseteq G$ is a finite set, closed under forming inverse elements, define $$\Delta(S)= \vert S\vert\cdot e -\sum_{s\in S} s \in \R[G]_h,$$ and call it the {\it Laplace operator} defined by $S$. Here, $e$ denotes the identity element of $G$. 

\begin{Thm}[Ozawa \cite{oz}]
Let $G$ be finitely generated by $S$. Then $G$ has {\it Kazhdan's Property (T)} if and only if $$\Delta(S)^2-\varepsilon \cdot\Delta(S) \in \Sigma^2\R[G]$$ for some $\varepsilon >0.$
\end{Thm}

\noindent
Kazhdan's Property (T) is an abstract property, introduced by Kazhdan in the 1960's, to prove that certain lattices are finitely generated. It is also important when examining random walks on Cayley graphs, and algorithms to produce random group elements \cite{lu, pt}. Ozawa's result is in particular surprising since all known definitions of property (T) refer to the class of unitary representations of the group on Hilbert space. In view of the available semidefinite programming methods for sums of squares, it opens the way for algorithmic approaches towards property (T):

\begin{Thm}[Netzer \& Thom \cite{kazhdan}]
For $G= {\rm SL}_3(\Z)$ (with canonical generating set $S$ of elementary matrices and their inverses) we have $$\Delta(S)^2- \frac16 \Delta(S)\in\Sigma^2\R[G].$$
\end{Thm}

\noindent
The sums of squares representation from this theorem was found numerically, with semidefinite programming via Gram matrices, as described in Section \ref{opt}. The (inexact) numerical solution was then transformed into an exact proof, with an additional argument. Although it was known that ${\rm SL}_3(\Z)$ has property (T), it is interesting that Ozawa's result is of real computational relevance. Also surprising is the explicit factor of $\varepsilon=\frac16$, which is by far larger than any previously known value. This is indeed relevant, since $\varepsilon$ gives information about the rate of convergence of the product replacement algorithm for abelian groups (see for example \cite{lu}).


\begin{bibdiv}
\begin{biblist}


\bib{aleknetzthom}{article}{
   author={Alekseev, V.}
   author={Netzer, T.},
   author={Thom, A.},
   title={Quadratic Modules, $C^*$-Algebras and Free Convexity},
   year={2016},
   journal={Preprint},
}

\bib{artin}{article}{
    AUTHOR = {Artin, E.},
     TITLE = {\"{U}ber die {Z}erlegung definiter {F}unktionen in {Q}uadrate},
   JOURNAL = {Abh. Math. Sem. Univ. Hamburg},
  FJOURNAL = {Abhandlungen aus dem Mathematischen Seminar der Universit\"at
              Hamburg},
    VOLUME = {5},
      YEAR = {1927},
    NUMBER = {1},
     PAGES = {100--115},
}

\bib{pt}{book}{
    AUTHOR = {Bekka, B.},
    AUTHOR={de la Harpe, P.},
    AUTHOR={Valette, A.},
     TITLE = {Kazhdan's property ({T})},
    SERIES = {New Mathematical Monographs},
    VOLUME = {11},
 PUBLISHER = {Cambridge University Press, Cambridge},
      YEAR = {2008},
     PAGES = {xiv+472},
}

\bib{blekh}{article}{
    AUTHOR = {Blekherman, G.},
     TITLE = {There are significantly more nonnegative polynomials than sums
              of squares},
   JOURNAL = {Israel J. Math.},
  FJOURNAL = {Israel Journal of Mathematics},
    VOLUME = {153},
      YEAR = {2006},
     PAGES = {355--380},
 }


\bib{par}{book}{
     TITLE = {Semidefinite optimization and convex algebraic geometry},
    SERIES = {MOS-SIAM Series on Optimization},
    VOLUME = {13},
    EDITOR = {Blekherman, G.},
    editor={Parrilo, P. A.}, 
    editor={Thomas, R. R.},
 PUBLISHER = {Society for Industrial and Applied Mathematics (SIAM),
              Philadelphia, PA; Mathematical Optimization Society,
              Philadelphia, PA},
      YEAR = {2013},
     PAGES = {xx+476},
 }


\bib{bcr}{book} {
    AUTHOR = {Bochnak, J.},
     AUTHOR = {Coste, M.},
      AUTHOR = {Roy, M.-F.},
     TITLE = {Real algebraic geometry},
    SERIES = {Ergebnisse der Mathematik und ihrer Grenzgebiete (3) [Results
              in Mathematics and Related Areas (3)]},
    VOLUME = {36},
      NOTE = {Translated from the 1987 French original,
              Revised by the authors},
 PUBLISHER = {Springer-Verlag, Berlin},
      YEAR = {1998},
     PAGES = {x+430},
}


\bib{br}{article}{
    AUTHOR = {Br{\"a}nd{\'e}n, P.},
     TITLE = {Hyperbolicity cones of elementary symmetric polynomials are
              spectrahedral},
   JOURNAL = {Optim. Lett.},
  FJOURNAL = {Optimization Letters},
    VOLUME = {8},
      YEAR = {2014},
    NUMBER = {5},
     PAGES = {1773--1782},
}

\bib{he0}{article}{
    AUTHOR = {Helton, J. W.},
     TITLE = {``{P}ositive'' noncommutative polynomials are sums of squares},
   JOURNAL = {Ann. of Math. (2)},
  FJOURNAL = {Annals of Mathematics. Second Series},
    VOLUME = {156},
      YEAR = {2002},
    NUMBER = {2},
     PAGES = {675--694},
  }




\bib{he1}{incollection}{
    AUTHOR = {Helton, J. W.},
     AUTHOR = {Klep, I.},
      AUTHOR = {McCullough, S.},
     TITLE = {Free convex algebraic geometry},
 BOOKTITLE = {Semidefinite optimization and convex algebraic geometry},
    SERIES = {MOS-SIAM Ser. Optim.},
    VOLUME = {13},
     PAGES = {341--405},
 PUBLISHER = {SIAM, Philadelphia, PA},
      YEAR = {2013},
}



\bib{he2}{article}{
    AUTHOR = {Helton, J. W.},
     AUTHOR = {Klep, I.},
      AUTHOR = {McCullough, S.},
     TITLE = {The convex {P}ositivstellensatz in a free algebra},
   JOURNAL = {Adv. Math.},
  FJOURNAL = {Advances in Mathematics},
    VOLUME = {231},
      YEAR = {2012},
    NUMBER = {1},
     PAGES = {516--534},
}


\bib{he4}{article}{
    AUTHOR = {Helton, J. W.},
    AUTHOR={McCullough, S.},
     TITLE = {Every convex free basic semi-algebraic set has an {LMI}
              representation},
   JOURNAL = {Ann. of Math. (2)},
  FJOURNAL = {Annals of Mathematics. Second Series},
    VOLUME = {176},
      YEAR = {2012},
    NUMBER = {2},
     PAGES = {979--1013},
}


\bib{he5}{article}{
    AUTHOR = {Helton, J. W.},
     AUTHOR = {McCullough, S.},
      AUTHOR = {Putinar,  M.},
     TITLE = {A non-commutative {P}ositivstellensatz on isometries},
   JOURNAL = {J. Reine Angew. Math.},
  FJOURNAL = {Journal f\"ur die Reine und Angewandte Mathematik},
    VOLUME = {568},
      YEAR = {2004},
     PAGES = {71--80},
}


\bib{heni}{article}{
    AUTHOR = {Helton, J. W.},
    AUTHOR={Nie, J.},
     TITLE = {Semidefinite representation of convex sets},
   JOURNAL = {Math. Program.},
  FJOURNAL = {Mathematical Programming. A Publication of the Mathematical
              Programming Society},
    VOLUME = {122},
      YEAR = {2010},
    NUMBER = {1, Ser. A},
     PAGES = {21--64},
 }


\bib{heni2}{article}{
    AUTHOR = {Helton, J. W.},
    AUTHOR={Nie, J.},
     TITLE = {Sufficient and necessary conditions for semidefinite
              representability of convex hulls and sets},
   JOURNAL = {SIAM J. Optim.},
  FJOURNAL = {SIAM Journal on Optimization},
    VOLUME = {20},
      YEAR = {2009},
    NUMBER = {2},
     PAGES = {759--791},
 }

\bib{hevi}{article}{
    AUTHOR = {Helton, J. W.},
    AUTHOR={Vinnikov, V.},
     TITLE = {Linear matrix inequality representation of sets},
   JOURNAL = {Comm. Pure Appl. Math.},
  FJOURNAL = {Communications on Pure and Applied Mathematics},
    VOLUME = {60},
      YEAR = {2007},
    NUMBER = {5},
     PAGES = {654--674},
}

\bib{km}{article}{
    AUTHOR = {Kuhlmann, S.},
    AUTHOR={Marshall, M.},
     TITLE = {Positivity, sums of squares and the multi-dimensional moment
              problem},
   JOURNAL = {Trans. Amer. Math. Soc.},
  FJOURNAL = {Transactions of the American Mathematical Society},
    VOLUME = {354},
      YEAR = {2002},
    NUMBER = {11},
     PAGES = {4285--4301 (electronic)},
}


\bib{kms}{article}{
    AUTHOR = {Kuhlmann, S.},
    AUTHOR= {Marshall, M.},
    AUTHOR={Schwartz, N.},
     TITLE = {Positivity, sums of squares and the multi-dimensional moment
              problem. {II}},
   JOURNAL = {Adv. Geom.},
  FJOURNAL = {Advances in Geometry},
    VOLUME = {5},
      YEAR = {2005},
    NUMBER = {4},
     PAGES = {583--606},
  }


\bib{kummer}{article}{
	AUTHOR={Kummer, M.},
	TITLE={Determinantal Representations and B\'{e}zoutians},
	JOURNAL={Math. Z.},
	note={to appear},
	}

\bib{las}{article}{
    AUTHOR = {Lasserre, J. B.},
     TITLE = {Convex sets with semidefinite representation},
   JOURNAL = {Math. Program.},
  FJOURNAL = {Mathematical Programming. A Publication of the Mathematical
              Programming Society},
    VOLUME = {120},
      YEAR = {2009},
    NUMBER = {2, Ser. A},
     PAGES = {457--477},
 
}

\bib{las2}{article}{
    AUTHOR = {Lasserre, J. B.},
     TITLE = {Global optimization with polynomials and the problem of
              moments},
   JOURNAL = {SIAM J. Optim.},
  FJOURNAL = {SIAM Journal on Optimization},
    VOLUME = {11},
      YEAR = {2000/01},
    NUMBER = {3},
     PAGES = {796--817},
}

\bib{lof}{article}{
    AUTHOR = {L{\"o}fberg, J.,},
     TITLE = {Pre- and post-processing sum-of-squares programs in practice},
   JOURNAL = {IEEE Trans. Automat. Control},
  FJOURNAL = {Institute of Electrical and Electronics Engineers.
              Transactions on Automatic Control},
    VOLUME = {54},
      YEAR = {2009},
    NUMBER = {5},
     PAGES = {1007--1011},
}


\bib{lu}{article}{
    AUTHOR = {Lubotzky, A.},
    AUTHOR={Pak, I.},
     TITLE = {The product replacement algorithm and {K}azhdan's property
              ({T})},
   JOURNAL = {J. Amer. Math. Soc.},
  FJOURNAL = {Journal of the American Mathematical Society},
    VOLUME = {14},
      YEAR = {2001},
    NUMBER = {2},
     PAGES = {347--363 (electronic)},
}

\bib{ma}{book}{
    AUTHOR = {Marshall, M.},
     TITLE = {Positive polynomials and sums of squares},
    SERIES = {Mathematical Surveys and Monographs},
    VOLUME = {146},
 PUBLISHER = {American Mathematical Society, Providence, RI},
      YEAR = {2008},
     PAGES = {xii+187},
      ISBN = {978-0-8218-4402-1; 0-8218-4402-4},
   MRCLASS = {13J30 (14P10 44A60 90C22 90C26)},
  MRNUMBER = {2383959 (2009a:13044)},
MRREVIEWER = {Markus Schweighofer},
       DOI = {10.1090/surv/146},
       URL = {http://dx.doi.org/10.1090/surv/146},
}


\bib{strip}{article}{
    AUTHOR = {Marshall, M.},
     TITLE = {Polynomials non-negative on a strip},
   JOURNAL = {Proc. Amer. Math. Soc.},
  FJOURNAL = {Proceedings of the American Mathematical Society},
    VOLUME = {138},
      YEAR = {2010},
    NUMBER = {5},
     PAGES = {1559--1567},
}

\bib{motz}{incollection}{
    AUTHOR = {Motzkin, T. S.},
     TITLE = {The arithmetic-geometric inequality},
 BOOKTITLE = {Inequalities ({P}roc. {S}ympos. {W}right-{P}atterson {A}ir
              {F}orce {B}ase, {O}hio, 1965)},
     PAGES = {205--224},
 PUBLISHER = {Academic Press, New York},
      YEAR = {1967},
}



\bib{netznote}{article}{
	author={Netzer, T.},
	title={Reelle Algebraische Geometrie},
	note={Lecture Notes. \url{https://algebra-mathematics.uibk.ac.at/images/documents/teaching/tim_netzer/RAG.pdf}},
}

\bib{habil}{article}{
	author={Netzer, T.},
	title={Spectrahedra and their Shadows},
	note={Habilitation Thesis. \url{https://algebra-mathematics.uibk.ac.at/images/documents/other/Habilitationsschrift.pdf}}
	
}




\bib{neplth}{article}{
    AUTHOR = {Netzer, T.},
    AUTHOR={Plaumann, D.},
    AUTHOR={Thom, A.},
     TITLE = {Determinantal representations and the {H}ermite matrix},
   JOURNAL = {Michigan Math. J.},
  FJOURNAL = {Michigan Mathematical Journal},
    VOLUME = {62},
      YEAR = {2013},
    NUMBER = {2},
     PAGES = {407--420},
   }

\bib{kazhdan}{article}{
   author={Netzer, T.},
   author={Thom, A.},
   title={Kazhdan's Property (T) via Semidefinite Optimization},
   journal={Experimental Mathematics},
   volume={24},
   date={2015},
   number={2},
   pages={1--4},
}


\bib{cliff}{article}{
    AUTHOR = {Netzer, T.},
    AUTHOR={Thom, A.},
     TITLE = {Hyperbolic polynomials and generalized {C}lifford algebras},
   JOURNAL = {Discrete Comput. Geom.},
  FJOURNAL = {Discrete \& Computational Geometry. An International Journal
              of Mathematics and Computer Science},
    VOLUME = {51},
      YEAR = {2014},
    NUMBER = {4},
     PAGES = {802--814},
 }



\bib{schw1}{article}{
    AUTHOR = {Nie, J.},
    AUTHOR={Schweighofer, M.},
     TITLE = {On the complexity of {P}utinar's {P}ositivstellensatz},
   JOURNAL = {J. Complexity},
  FJOURNAL = {Journal of Complexity},
    VOLUME = {23},
      YEAR = {2007},
    NUMBER = {1},
     PAGES = {135--150},
 }




\bib{oz}{article}{
    AUTHOR = {Ozawa, N.},
     TITLE = {Noncommutative real algebraic geometry of {K}azhdan's property
              ({T})},
   JOURNAL = {J. Inst. Math. Jussieu},
  FJOURNAL = {Journal of the Institute of Mathematics of Jussieu. JIMJ.
              Journal de l'Institut de Math\'ematiques de Jussieu},
    VOLUME = {15},
      YEAR = {2016},
    NUMBER = {1},
     PAGES = {85--90},
 }

\bib{pl2}{article}{
    AUTHOR = {Plaumann, D.},
    AUTHOR={Sinn, R.},
    AUTHOR={Speyer, D. E.},
    AUTHOR={Vinzant, C.},
     TITLE = {Computing {H}ermitian determinantal representations of
              hyperbolic curves},
   JOURNAL = {Internat. J. Algebra Comput.},
  FJOURNAL = {International Journal of Algebra and Computation},
    VOLUME = {25},
      YEAR = {2015},
    NUMBER = {8},
     PAGES = {1327--1336},
  }


\bib{pl1}{article}{
    AUTHOR = {Plaumann, D.},
    AUTHOR={Vinzant, C.},
     TITLE = {Determinantal representations of hyperbolic plane curves: an
              elementary approach},
   JOURNAL = {J. Symbolic Comput.},
  FJOURNAL = {Journal of Symbolic Computation},
    VOLUME = {57},
      YEAR = {2013},
     PAGES = {48--60},
}


\bib{pd}{book}{ 
    AUTHOR = {Prestel, A.},
     AUTHOR = {Delzell, Ch. N.},
     TITLE = {Positive polynomials},
    SERIES = {Springer Monographs in Mathematics},
      NOTE = {From Hilbert's 17th problem to real algebra},
 PUBLISHER = {Springer-Verlag, Berlin},
      YEAR = {2001},
     PAGES = {viii+267},
 }


\bib{rago}{article}{
    AUTHOR = {Ramana, M.},
    AUTHOR={Goldman, A. J.},
     TITLE = {Some geometric results in semidefinite programming},
   JOURNAL = {J. Global Optim.},
  FJOURNAL = {Journal of Global Optimization. An International Journal
              Dealing with Theoretical and Computational Aspects of Seeking
              Global Optima and Their Applications in Science, Management
              and Engineering},
    VOLUME = {7},
      YEAR = {1995},
    NUMBER = {1},
     PAGES = {33--50},
 }

\bib{ren}{article}{
    AUTHOR = {Renegar, J.},
     TITLE = {Hyperbolic programs, and their derivative relaxations},
   JOURNAL = {Found. Comput. Math.},
  FJOURNAL = {Foundations of Computational Mathematics. The Journal of the
              Society for the Foundations of Computational Mathematics},
    VOLUME = {6},
      YEAR = {2006},
    NUMBER = {1},
     PAGES = {59--79},
  }

\bib{sch1}{article}{
    AUTHOR = {Scheiderer, C.},
     TITLE = {Sums of squares of regular functions on real algebraic
              varieties},
   JOURNAL = {Trans. Amer. Math. Soc.},
  FJOURNAL = {Transactions of the American Mathematical Society},
    VOLUME = {352},
      YEAR = {2000},
    NUMBER = {3},
     PAGES = {1039--1069},
}


\bib{sch4}{article}{
    AUTHOR = {Scheiderer, C.},
     TITLE = {Distinguished representations of non-negative polynomials},
   JOURNAL = {J. Algebra},
  FJOURNAL = {Journal of Algebra},
    VOLUME = {289},
      YEAR = {2005},
    NUMBER = {2},
     PAGES = {558--573},
   }

\bib{sch2}{article}{
	AUTHOR={Scheiderer, C.},
	TITLE={Semidefinite representation for convex hulls of real algebraic curves},
	JOURNAL={Preprint},
	YEAR={2012},
}





\bib{schm}{article}{
   author={Schm{\"u}dgen, K.},
   title={Noncommutative real algebraic geometry---some basic concepts and
   first ideas},
   conference={
      title={Emerging applications of algebraic geometry},
   },
   book={
      series={IMA Vol. Math. Appl.},
      volume={149},
      publisher={Springer, New York},
   },
   date={2009},
   pages={325--350},
}




\bib{schm4}{article}{
    AUTHOR = {Schm{\"u}dgen, K.},
     TITLE = {The {$K$}-moment problem for compact semi-algebraic sets},
   JOURNAL = {Math. Ann.},
  FJOURNAL = {Mathematische Annalen},
    VOLUME = {289},
      YEAR = {1991},
    NUMBER = {2},
     PAGES = {203--206},

}

\bib{schw2}{article}{
    AUTHOR = {Schweighofer, M.},
     TITLE = {On the complexity of {S}chm\"udgen's positivstellensatz},
   JOURNAL = {J. Complexity},
  FJOURNAL = {Journal of Complexity},
    VOLUME = {20},
      YEAR = {2004},
    NUMBER = {4},
     PAGES = {529--543},

}

\bib{sinn}{article}{
    AUTHOR = {Sinn, R.},
     TITLE = {Algebraic boundaries of {$SO(2)$}-orbitopes},
   JOURNAL = {Discrete Comput. Geom.},
  FJOURNAL = {Discrete \& Computational Geometry. An International Journal
              of Mathematics and Computer Science},
    VOLUME = {50},
      YEAR = {2013},
    NUMBER = {1},
     PAGES = {219--235},
}


\bib{hbsdp}{book}{
     TITLE = {Handbook of semidefinite programming},
    SERIES = {International Series in Operations Research \& Management
              Science, 27},
    EDITOR = {Wolkowicz, H.},
    EDITOR={Saigal, R,},
    EDITOR={Vandenberghe, L.},
      NOTE = {Theory, algorithms, and applications},
 PUBLISHER = {Kluwer Academic Publishers, Boston, MA},
      YEAR = {2000},
     PAGES = {xxviii+654},

}


\end{biblist}
\end{bibdiv}
\end{document}